\documentclass[reqno]{amsart}
\usepackage{amssymb,enumerate,bbm}

\numberwithin{equation}{section}
\newtheorem{theorem}[equation]{Theorem}
\newtheorem{proposition}[equation]{Proposition}

\newtheorem{lemma}[equation]{Lemma}
\newtheorem{remark}[equation]{Remark}

\newcommand{\N}{\mathbbmss{N}}
\newcommand{\Z}{\mathbbmss{Z}}
\newcommand{\Q}{\mathbbmss{Q}}
\newcommand{\ndiv}{\not|\,}
\renewcommand{\mid}{\,|\,}

\title[{Factoring polynomials in $\Z[[x]]$}]{Factoring polynomials in the ring of formal power series over $\Z$}

\author{Daniel Birmajer}
\thanks{D. Birmajer would like to acknowledge the support and hospitality of the Universidad Nacional de San Luis, San Luis, Argentina}
\address{Department of Mathematics\\ Nazareth College\\ 4245 East Ave.\\ Rochester, NY 14618}
\email{abirmaj6@naz.edu}
\author{Juan B. Gil}
\address{Penn State Altoona\\ 3000 Ivyside Park\\ Altoona, PA 16601.}
\email{jgil@psu.edu}
\author{Michael Weiner}
\address{Penn State Altoona\\ 3000 Ivyside Park\\ Altoona, PA 16601.}
\email{mdw8@psu.edu}

\begin{document}

\begin{abstract}
We consider polynomials with integer coefficients and discuss their factorization properties in $\Z[[x]]$, the ring of formal power series over $\Z$. We treat polynomials of arbitrary degree and give sufficient conditions for their reducibility as power series. Moreover, if a polynomial is reducible over $\Z[[x]]$, we provide an explicit factorization algorithm. For polynomials whose constant term is a prime power, our study leads to the discussion of $p$-adic integers.
\end{abstract}

\maketitle

%%%%%%%%%%%%%%%%%%%%%%%%%%%%%%%%%%%%%%%%%%%%%%%%%%%%%%%%%%%%%%%%
\section{Introduction}

In this paper we consider polynomials with integer coefficients and discuss their factorization as elements of $\Z[[x]]$, the ring of formal power series over $\Z$. We treat polynomials of arbitrary degree and give sufficient conditions for their reducibility in $\Z[[x]]$. For polynomials of degree two or three these conditions are also necessary. 
\par
If the constant term of the polynomial is not a prime power, the reducibility discussion is straightforward.  We briefly address these cases in Section~\ref{sec:preliminaries}. On the other hand, if the constant term of the polynomial is a nontrivial prime power, say $p^n$ with $p$ prime and $n\ge 2$, the question of reducibility in $\Z[[x]]$ leads in some cases (when $n$ is no greater than twice the $p$-adic valuation of the linear coefficient) to the discussion of $p$-adic integers. In this context, our main result is that, if such a polynomial has a root in $\Z_p$ whose $p$-adic valuation is positive, then it is reducible in $\Z[[x]]$. This particular case is presented in Section~\ref{sec:Main}. Our proofs are constructive and provide explicit factorization algorithms. 
\par
It is important to note that irreducible elements in $\Z[x]$ and in $\Z[[x]]$ are, in general, unrelated. For instance, $6+x+x^2$ is irreducible in $\Z[x]$ but can be factored in $\Z[[x]]$, while $2+7x+3x^2$ is irreducible in $\Z[[x]]$ but equals $(2+x)(1+3x)$ as a polynomial. Observe that the latter is not a proper factorization in $\Z[[x]]$ since $1+3x$ is an invertible element.  More examples and a few remarks concerning the hypotheses of our main theorem are given in Section~\ref{sec:further}. For illustrative purposes, we explicitly discuss our results in the context of cubic polynomials.
\par
A first study of the factorization theory of quadratic polynomials in $\Z[[x]]$ was presented in \cite{BGW}.  The results of this paper contain and expand those obtained in the quadratic case to polynomials in $\Z[x]$ of arbitrary degree. All the needed material regarding the ring $\Z_p$ of $p$-adic integers can be found in \cite{Katok, Serre}.

%%%%%%%%%%%%%%%%%%%%%%%%%%%%%%%%%%%%%%%%%%%%%%%%%%%%%%%%%%%%%%%%
\section{Preliminaries}
\label{sec:preliminaries}

We start by reviewing some of the basic properties of the factorization theory of power series over $\Z$. For a more extensive discussion, we refer the reader to \cite{BiGi}. 

Let $f(x)=f_0+f_1 x+\cdots+f_d x^d$ be a polynomial with integer coefficients. It is easy to check that $f(x)$ is invertible in $\Z[[x]]$ if and only if $f_0=\pm 1$.

We say that $f(x)$ is reducible in $\Z[[x]]$ if there exist power series,
\begin{equation*}
A(x)=\sum_{k=0}^\infty a_k x^k \text{ and } B(x)=\sum_{k=0}^\infty b_k x^k \;\text{ with } a_k,\;b_k\in\Z,
\end{equation*}
such that $a_0\not=\pm 1$, $b_0\not=\pm 1$, and $f(x)=A(x)B(x)$. 

A proof of the following basic proposition can be found in \cite{BiGi}.

\begin{proposition} 
Let $f(x)$ be a non-invertible polynomial in $\Z[[x]]$ .
\begin{enumerate}[$(a)$]
\item If $f_0$ is prime, then $f(x)$ is irreducible in $\Z[[x]]$.
\item If $f_0$ is not a prime power, then $f(x)$ is reducible in $\Z[[x]]$.
\item If $f_0=p^n$ with $p$ prime, $n\ge 2$, and $p\ndiv f_1$, then $f(x)$ is irreducible in $\Z[[x]]$.
\end{enumerate}
\end{proposition}

As an immediate consequence, we have:  

\begin{proposition}[Linear polynomials]
A polynomial 
  \begin{equation*}
    f(x)=p^n+f_1 x
  \end{equation*}
with $n\ge 2$ is reducible in $\Z[[x]]$ if and only if $p\mid f_1$.
\end{proposition}

It remains to examine the reducibility of nonlinear polynomials of the form
\begin{equation}\label{def:polynomial}
f(x)=p^n + p^m\gamma_1 x + \gamma_2 x^2 +\cdots+ \gamma_dx^d, \quad \gamma_d\not=0,
\end{equation}
with $p$ prime, $n\ge 2$, $m\geq 1$, $\gcd(p,\gamma_1)=1$ or $\gamma_1=0$, and $\gcd(p,\gamma_2,\dots,\gamma_d)=1$. 

As it turns out, the reducibility of such polynomials depends on the relation between the parameters $n$ and $m$ in \eqref{def:polynomial} and on their factorization properties as elements of $\Z_p[x]$. Accordingly, we divide our study into two cases: The case when $n>2m$, see Proposition \ref{case:n>2m}, and the case when $n\le 2m$, which is more involved and will be investigated in the next section. 

\begin{remark}
If a polynomial of the form \eqref{def:polynomial} has no linear term, i.e. if $\gamma_1=0$, we can assume $m$ as large as needed and can therefore discuss its reducibility as for the case when $n\le 2m$.
\end{remark}

\begin{proposition}\label{case:n>2m}
If $n>2m$ and $\gcd(p,\gamma_1)=1$, then the polynomial $f(x)$ in \eqref{def:polynomial} is reducible in both $\Z[[x]]$ and $\Z_p[x]$.
\end{proposition}
\begin{proof}
First of all, observe that $f(p^n)\equiv 0\pmod{p^n}$, and $f'(p^n)\equiv 0\pmod{p^{m}}$ because $n>m$. Moreover, since $n>2m\ge m+1$ and $\gcd(p,\gamma_1)=1$, we have $f'(p^n)\not\equiv 0\pmod{p^{m+1}}$. Then, by Hensel's lemma, $p^n$ lifts to a root of $f(x)$ in $\Z_p$. Thus $f(x)$ is reducible in $\Z_p[x]$.

To show the reducibility of $f(x)$ in $\Z[[x]]$, we give an inductive procedure to find $a_k,b_k\in\Z$ such that
\begin{equation*}
 f(x)=\big(p^m+a_1 x+a_2 x^2+\cdots\big)\big(p^{n-m}+b_1 x+b_2 x^2+\cdots\big).
\end{equation*}
First of all, observe that $n>2m$ implies $n-m>m$.  As a first step, we need
\begin{align*}
p^m \gamma_1 &= p^m (b_1+ p^{n-2m}a_1)\\ 
\gamma_2 &= p^m s_2 +a_1(\gamma_1-p^{n-2m}a_1)
\end{align*}
with $s_2=b_2+p^{n-2m}a_2$. If we let $h(x)=\gamma_2-\gamma_1 x+p^{n-2m}x^2$, then solving the last equation is equivalent to finding $a_1\in\Z$ such that $h(a_1)\equiv 0 \!\pmod{p^m}$. Since the discriminant of $h(x)$ is a square mod $p$, there exists $r\in\Z_p$ such that $h(r)=0$. We let $a_1\in\Z$ be the reduction of $r$ mod $p^m$ and let $s_2=h(a_1)/p^m$. 

With the notation $s_j=b_j+p^{n-2m}a_j$, the rest of the equations become
\begin{align*}
\gamma_3& = p^m s_3 + a_2(\gamma_1-2p^{n-2m}a_1) + a_1s_2 \\
\gamma_4 &= p^m s_4 +a_3(\gamma_1-2p^{n-2m}a_1)+ a_1s_3 + a_2b_2 \\[-1ex] 
& \;\; \vdots \hspace*{9em} \vdots \\[-1ex]
\gamma_d &= p^m s_d +a_{d-1}(\gamma_1-2p^{n-2m}a_1)+ a_1s_{d-1}+a_2 b_{d-2}+\cdots+ a_{d-2} b_2 \\
0 &= p^m s_{d+1} +a_{d}(\gamma_1-2p^{n-2m}a_1)+ a_1s_{d}+a_2 b_{d-1}+\cdots+ a_{d-1} b_2 \\[-1ex]
& \;\; \vdots \hspace*{9em} \vdots \\[-1ex]
0 &= p^m s_{k+1} +a_k(\gamma_1-2p^{n-2m}a_1)+ a_1 s_k + a_2 b_{k-1}+\cdots+ a_{k-1} b_2 \\[-1ex]
& \;\; \vdots \hspace*{9em} \vdots
\end{align*}
Since $\gamma_1-2p^{n-2m}a_1$ is not divisible by $p$, there are integers $a_2$ and $s_3$ such that the equation for $\gamma_3$ is satisfied. With the same argument, one can proceed inductively to solve for every pair $a_k, s_{k+1}\in\Z$.
\end{proof}

In the next section, we will make use of the following elementary result:
\begin{lemma}[Theorem~1.42 in \cite{Katok}] \label{rootinZp}
A polynomial with integer coefficients has a root in $\Z_p$ if and only if it has an integer root modulo $p^k$ for any $k\ge 1$.
\end{lemma}

%%%%%%%%%%%%%%%%%%%%%%%%%%%%%%%%%%%%%%%%%%%%%%%%%%%%%%%%%%%%%%%%
\section{Factorization in the presence of a $p$-adic root}
\label{sec:Main}

In this section, we finish our discussion of reducibility in $\Z[[x]]$ for polynomials with integer coefficients. Our results rely on the existence of a root $\varrho \in p\Z_p$, and the nature of the factorization depends on the multiplicity of $\varrho$ and on its $p$-adic valuation, denoted by $v_p(\varrho)$. As we will show later, the existence of a root in $p\Z_p$ is a necessary condition for quadratic and cubic polynomials. 

If $\varrho\in p\Z_p$ is a multiple root of $f(x)$, then it is reducible with factors in $\Z[x]$.
 
\begin{proposition}\label{prop:doubleroot}
Let $f(x)$ be a polynomial of degree $d\ge 2$. If $f(x)$ has a multiple root $\varrho\in\Z_p$ with $v_p(\varrho)=\ell\ge 1$, then $f(x)$ admits a proper factorization 
\begin{equation*}
f(x)=G(x) \cdot   f_{\textup{red}}(x)
\end{equation*}
where $G(x)=\gcd(f, f^\prime)$ in $\Z[x]$, and $f_{\textup{red}}(x)=f(x)/G(x)$.
 \end{proposition}
\begin{proof}
Let $\varrho\in\Z_p$ with $v_p(\varrho)=\ell$ be the multiple root of $f(x)$. 

Let $G(x)=\gcd(f, f^\prime)\in \Z[x]$. Then $f_{\text{red}}(x)=f(x)/G(x)$ is an element of $\Z[x]$ and $f(x)= G(x) f_{\text{red}}(x)$. This factorization is also valid in $\Q_p[x]$, and every root of $f(x)$ is a simple root of $f_{\text{red}}(x)$. Thus, together with $f_{\text{red}}(\varrho)=0$, we must also have $G(\varrho)=0$. Hence $p^\ell$ divides both $f_{\text{red}}(0)$ and $G(0)$ giving that the factorization $f(x)= G(x) f_{\text{red}}(x)$ is a proper factorization of $f(x)$ in $\Z[[x]]$.
\end{proof}

We now tackle the case when $f(x)$ has a simple root in $p\Z_p$.

\begin{theorem}\label{case:n<=2m}
Let $f(x)$ be a polynomial of the form \eqref{def:polynomial} with $n\le 2m$ and $d\ge 2$. If $f(x)$ has a simple root $\varrho\in\Z_p$ with $v_p(\varrho)=\ell\ge 1$, then $f(x)$ admits a factorization 
\begin{equation}\label{factorization}
 f(x)=\big(p^\ell+a_1 x+a_2 x^2+\cdots\big)\big(p^{n-\ell}+b_1 x+b_2 x^2+\cdots\big) \text{ in } \Z[[x]].
\end{equation}
\end{theorem}
\begin{proof}
Let $\varrho\in\Z_p$, with $v_p(\varrho)=\ell$, be the root of $f(x)$. Since $n\le 2m$ and
\begin{equation*}
 p^n + p^m\gamma_1 \varrho + \gamma_2 \varrho^2 +\dotsb +\gamma_d \varrho^d = 0,
\end{equation*}
we must have $\ell\le m$ and $2\ell\le n$.

The goal is to find coefficients $a_k, b_k\in\Z$ such that $f(x)$ factors as in \eqref{factorization}.  

First, consider the polynomial given by $g(-x)=\frac{x^d}{p^{2\ell}}f(p^\ell/x)$. Thus
\begin{align*}
  g(x) =p^{(d-2)\ell}\gamma_d &- p^{(d-3)\ell}\gamma_{d-1}x + p^{(d-4)\ell}\gamma_{d-2}x^2+\cdots \\
  &\cdots+ (-1)^{d-2} \gamma_2 x^{d-2} + (-1)^{d-1} p^{m-\ell}\gamma_1 x^{d-1} + (-1)^{d}  p^{n-2\ell} x^d.
\end{align*}
We set $g_d(x)=g(x)$ and define recursively
\begin{equation*}
  g_k(x)=\frac{p^{(k-1)\ell}\gamma_{k+1}-g_{k+1}(x)}{x} \;\text{ for } k=(d-1), (d-2), \dots, 2.
\end{equation*}
In particular, $g_2(x)=\gamma_2 - p^{m-\ell}\gamma_1 x + p^{n-2\ell} x^2$. Now, if we solve for $g_{k+1}(x)$ and take its derivative, we obtain \begin{equation}\label{DerivativeRelation}
  g_{k+1}'(x) = -\big(g_k(x)+xg_k'(x)\big) \text{ for every } 2\le k \le d-1.
\end{equation}

Note that if $r=-p^{\ell}/\varrho \in\Z_p$, then $v_p(r)=0$ and $g_d(r)=0$. Moreover, since $f'(\varrho)\not=0$, we have $g_d'(r)\not=0$. Let $\theta=v_p(g_d'(r))$. 

Our factoring starts by letting $a_1\in\Z$ be a reduction of $r$ mod $p^{(2d-3)\ell+2\theta}$. Then
\begin{equation*}
g_d(a_1) \equiv 0\!\!\! \pmod{p^{(2d-3)\ell+2\theta}} \;\text{ and }\; v_p(g_d'(a_1)) =\theta. 
\end{equation*}
Since $g_d(a_1)= p^{(d-2)\ell}\gamma_{d} - a_1g_{d-1}(a_1)$, we have that $p^{(d-2)\ell}\mid g_{d-1}(a_1)$. If we proceed iteratively, we obtain that
\begin{equation*}
  p^{(k-1)\ell}\mid g_{k}(a_1) \text{ for every } 2\le k \le d-1.
\end{equation*}

We set $a_2=a_3=\dots =a_{d-1}=0$, and define
\begin{equation*}
s_j=b_j+ p^{n-2\ell}a_j \text{ for every } j\in\N.
\end{equation*}
With these conventions, finding the factorization \eqref{factorization} is equivalent to solving in $\Z$ the system of equations:
\begin{equation}\label{GeneralSystem}
\begin{split}
p^m \gamma_1 &= p^\ell s_1\\ 
\gamma_2 &= p^\ell s_2 + a_1(s_1-p^{n-2\ell}a_1)\\
\gamma_3& = p^\ell s_3 + a_1s_2 \\[-1ex] 
& \;\; \vdots \\[-1ex]
\gamma_d &= p^\ell s_d + a_1s_{d-1} \\
0 &= p^\ell s_{d+1} +a_{d}b_1+ a_1b_d \\
0 &= p^\ell s_{d+2} +a_{d+1}b_1+ a_d s_2 + a_1b_{d+1}\\[-1ex]
& \;\; \vdots 
\end{split}
\end{equation}
The first equation then gives $s_1=p^{m-\ell}\gamma_1$, so the second equation becomes
\[ \gamma_2 = p^\ell s_2 + a_1(p^{m-\ell}\gamma_1-p^{n-2\ell}a_1). \]
Now, if we set $s_2 = p^{-\ell} g_2(a_1)$, this equation is also satisfied. In fact, if we set
\begin{equation}\label{eq:s_kg_k}
 s_k = p^{-(k-1)\ell} g_{k}(a_1) \;\text{ for } 2\le k\le d,
\end{equation}
the integers $a_1$, $s_1, s_2,\dots,s_d$ solve the first $d$ equations of the above system. At this point, we also have 
$b_1=p^{m-\ell}\gamma_1-p^{n-2\ell}a_1$ and $b_k=s_k$ for $k=2,\dots,d-1$. 

For $k\ge d$, we combine $a_kb_1$ with $a_1 b_k$, and use $g_2'(a_1)=-(p^{m-\ell}\gamma_1 - 2p^{n-2\ell}a_1)$ to rewrite each corresponding equation in \eqref{GeneralSystem} as
\begin{equation*}
 0 = p^\ell s_{k+1} - a_k g_2'(a_1) +a_1 s_k + \sum_{j=d}^{k-1} a_j b_{k+1-j}.
\end{equation*}

We proceed with an inductive algorithm to choose $a_{d+j}$ and $s_{d+1+j}$ for $j\ge 0$. This will then determine $b_{d+j}$, providing the desired factorization of $f(x)$. 

\subsection*{Base step} Choosing $a_d$ and $s_{d+1}$.
We consider the next block of $d-1$ consecutive equations in \eqref{GeneralSystem}. As mentioned above, these equations can be written as\footnote{Recall that $b_k=s_k$ for $k=2,\dots,d-1$}
\begin{equation}\label{eq:basecase}
\begin{split}
0 &= p^\ell s_{d+1} - a_{d}g_2^\prime(a_1)+a_1s_d\\
0 &= p^\ell s_{d+2} - a_{d+1}g_2^\prime(a_1)+a_ds_2+a_1s_{d+1}\\
0 &= p^\ell s_{d+3} - a_{d+2}g_2^\prime(a_1)+a_{d+1}s_2+a_ds_3+a_1s_{d+2}\\[-1ex]
   & \;\; \vdots \\[-1ex]
0 &= p^\ell s_{2d-1} - a_{2d-2}g_2^\prime(a_1)+a_{2d-3}s_2+\dots + a_{d}s_{d-1}+a_1s_{2d-2}.
\end{split}
\end{equation}

The first two equations combined give
\begin{align} \notag
0 &= p^{2\ell} s_{d+2} -p^\ell a_{d+1}g_2^\prime(a_1)+p^\ell a_d s_2+a_1(a_{d}g_2^\prime(a_1)-a_1s_d)\\ \notag
&=p^{2\ell} s_{d+2} -p^\ell a_{d+1}g_2^\prime(a_1)+a_{d}\bigr(g_2(a_1)+a_1g_2^\prime(a_1)\bigl)-a_1^2s_d\\ \label{eq:dplus2}
&=p^{2\ell} s_{d+2} -p^\ell a_{d+1}g_2^\prime(a_1)-a_{d}g_3^\prime(a_1)-a_1^2s_d
\end{align}
because $p^\ell s_2=g_2(a_1)$ and $g_2(a_1)+a_1g_2^\prime(a_1)=-g_3^\prime(a_1)$. If we now combine \eqref{eq:dplus2} with the third equation in \eqref{eq:basecase} and use, in addition, the identities $p^{2\ell} s_3=g_3(a_1)$ and $g_3(a_1)+a_1g_3^\prime(a_1)=-g_4^\prime(a_1)$, we then arrive at
\begin{equation*}
0 = p^{3\ell} s_{d+3} -p^{2\ell} a_{d+2}g_2^\prime(a_1)-a_{d+1}p^\ell g_3^\prime(a_1)-a_dg_4^\prime(a_1)+a_1^3 s_d.
\end{equation*}

Continuing this process, using the relations \eqref{DerivativeRelation} and \eqref{eq:s_kg_k} in each iteration, the system \eqref{eq:basecase} becomes
\begin{equation} \label{eq:BaseBlock}
\begin{split} 
0 &= p^\ell s_{d+1} - a_{d}g_2^\prime(a_1)+a_1s_d\\ 
0 &= p^{2\ell} s_{d+2} -p^\ell a_{d+1}g_2^\prime(a_1)-a_{d}g_3^\prime(a_1)-a_1^2s_d\\ 
0 &= p^{3\ell} s_{d+3} -p^{2\ell} a_{d+2}g_2^\prime(a_1)-a_{d+1}p^\ell g_3^\prime(a_1)-a_dg_4^\prime(a_1)+a_1^3 s_d\\[-1ex]
   & \;\; \vdots \\[-1ex] 
0 &=p^{(d-1)\ell}s_{2d-1}-p^{(d-2)\ell}a_{2d-2}g_2^\prime(a_1)-\dots -a_d g_d^\prime(a_1)+(-1)^{d}a_1^{d-1}s_d.
\end{split}
\end{equation}

We now set 
\begin{equation}\label{eq:tdDef}
 t_d=p^{(d-2)\ell}s_{2d-1}-p^{(d-3)\ell}a_{2d-2}g_2^\prime(a_1)-\dots -a_{d+1} g_{d-1}^\prime(a_1)
\end{equation}
and rewrite the last equation in \eqref{eq:BaseBlock} as
\begin{align} 
\label{eq:td-basic}
0 &=p^{\ell}t_d - a_d g_d^\prime(a_1)+(-1)^{d}a_1^{d-1}s_d \\
\label{eq:td-theta}
0 &=p^{\ell}t_{d}-p^\theta \big[a_d p^{-\theta}g_d^\prime(a_1)+(-1)^{d-1}a_1^{d-1}p^{-\theta}s_d\big].
\end{align}
Since $p^{(2d-3)\ell+2\theta}\mid g_d(a_1)$ and since $s_d=p^{-(d-1)\ell}g_d(a_1)$, we have that $p^{(d-2)\ell+2\theta}\mid s_d$. In particular, $p^{-\theta}s_d\in\Z$. Now, since $v_p(p^{-\theta}g_d^\prime(a_1))=0$, we can choose $a_d\in\Z$ such that the expression inside the brackets in \eqref{eq:td-theta} is divisible by $p^{(d-1)\ell+\theta}$. We then solve for $t_d$ and obtain that $t_d$ is divisible by $p^{(d-2)\ell+2\theta}$. Furthermore, from equation \eqref{eq:td-basic} we deduce that $a_d$ must be divisible by $p^{(d-2)\ell+\theta}$.

With $a_d$ having been chosen, we go back to the first equation of \eqref{eq:BaseBlock} and solve for $s_{d+1}$. 
It follows that, at the very least, $p^\theta\mid s_{d+1}$. Note that if $d=2$, then $s_{d+1}=s_{2d-1}=t_d$, so the choice of $t_d$ in equation \eqref{eq:td-theta} already determines $s_{d+1}$.

\subsection*{Inductive step} Choosing $a_{d+j}$, $t_{d+j}$ and $s_{d+1+j}$.

Consistent with the above definition of $t_d$, we now let
\begin{equation*}
 t_{d+j}=p^{(d-2)\ell}s_{2d-1+j}-p^{(d-3)\ell}a_{2d-2+j}g_2^\prime(a_1)-\cdots -a_{d+1+j} g_{d-1}^\prime(a_1).
\end{equation*}

Following the pattern in \eqref{eq:basecase}, the equation for $s_{2d}$ can then be written as
\begin{equation*}
 0=p^\ell s_{2d} - a_{2d-1}g_2^\prime(a_1)+a_{2d-2}s_2+\dots + a_{d+1}s_{d-1}+a_d b_d+a_1s_{2d-1}.
\end{equation*}
If we multiply both sides of this equation by $p^{(d-2)\ell}$ and use the relations \eqref{DerivativeRelation}, \eqref{eq:s_kg_k}, and \eqref{eq:tdDef}, we obtain the equivalent equation
\begin{equation*}
 0=p^\ell t_{d+1} - a_{d+1} g_d^\prime(a_1)+a_1t_d+ p^{(d-2)\ell}a_db_d.
\end{equation*}

Repeating this process iteratively, for every $j\ge 1$, the equation for $s_{2d+j-1}$ in the system \eqref{GeneralSystem} can be reduced to
\begin{equation}\label{eq:j_equation}
0 =p^{\ell}t_{d+j} - a_{d+j} g_d^\prime(a_1)+a_1t_{d+j-1}+ p^{(d-2)\ell} R_{d+j-1},
\end{equation}
where $R_{d+j-1}=a_{d}b_{d+j-1}+\dots+a_{d+j-1}b_{d}$.

Assume that the numbers $a_{d+i}$,  $t_{d+i}$ and $s_{d+1+i}$ have been chosen for all $i$ with $0\le i\le j-1$, in such a way that 
the corresponding equations are all satisfied, and such that
\begin{equation*}
 p^{(d-2)\ell+2\theta} \mid t_{d+i},\quad  p^{(d-2)\ell+\theta} \mid a_{d+i}, \quad p^{\theta} \mid s_{d+1+i}.
\end{equation*}
Note that we then have $b_{d+i}$ as well and, in particular, $p^{\theta} \mid b_{d+i}$ for all $0\le i\le j-1$. Thus each $R_{d+i}$ is determined and $p^{2\theta} \mid R_{d+i}$.

We now proceed to choose $a_{d+j}$ and $t_{d+j}$. Since $p^\theta$ divides $g_d^\prime(a_1)$, $t_{d+j-1}$, and $R_{d+j-1}$, we can rewrite \eqref{eq:j_equation} as
\begin{equation*}
0 =p^{\ell}t_{d+j} - p^\theta\big[a_{d+j} p^{-\theta}g_d^\prime(a_1)-a_1p^{-\theta}t_{d+j-1}- p^{(d-2)\ell-\theta} R_{d+j-1}\big]
\end{equation*}
without leaving $\Z$. Now, as for the base step, we can choose $a_{d+j}\in\Z$ such that the bracket is divisible by $p^{(d-1)\ell+\theta}$ (since $v_p(p^{-\theta}g_d^\prime(a_1))=0$). We then solve for $t_{d+j}$ and get a number which is divisible by $p^{(d-2)\ell+2\theta}$. From equation \eqref{eq:j_equation}, it follows that $a_{d+j}$ must be divisible by $p^{(d-2)\ell+\theta}$. Once again, if $d=2$ we can stop here since $s_{d+1+j}=s_{2d-1+j}=t_{d+j}$.

To choose $s_{d+1+j}$ we consider two cases. If $j\ge d-2$, the definition of $t_{j+2}$ gives
\begin{equation*}
 p^{(d-2)\ell}s_{d+1+j} = t_{j+2} + p^{(d-3)\ell}a_{d+j}g_2^\prime(a_1) + \cdots + a_{j+3} g_{d-1}^\prime(a_1).
\end{equation*}
Since the numbers $a_{j+3},\dots,a_{d+j}$ are all divisible by $p^{(d-2)\ell+\theta}$, we can solve for $s_{d+1+j}$ in $\Z$, and get that it is divisible by $p^\theta$. 

If $j<d-2$, then $d+1+j < 2d-1$ and we can use the equations in \eqref{eq:basecase} to solve for  $s_{d+1+j}$. More precisely, we consider the equation
\begin{equation*}
0 = p^\ell s_{d+j+1} - a_{d+j}g_2'(a_1)+a_{d+j-1}b_2+\dots + a_{d}b_{j+1}+a_1s_{d+j}.
\end{equation*}
Again, the numbers $a_d,\dots,a_{d+j}$ are all divisible by $p^{(d-2)\ell+\theta}$, in particular they are divisible by $p^{\ell+\theta}$. On the other hand, since $p^{(d-2)\ell+2\theta}$ divides $s_d$, an iteration of the equations in \eqref{eq:basecase} gives that each $s_{d+j}$ is divisible by $p^{(d-2-j)\ell+\theta}$ for every $1\le j \le d-2$, and therefore $p^{\ell+\theta}\mid s_{d+j}$ for $j<d-2$. That means, also in this case, we can solve for $s_{d+1+j}$ in $\Z$ and obtain $p^\theta \mid s_{d+1+j}$.

Altogether, we have found $a_{d+j}$, $t_{d+j}$, $s_{d+1+j}\in\Z$ satisfying the equations under consideration, and such that  
\begin{equation*}
 p^{(d-2)\ell+2\theta} \mid t_{d+j},\quad  p^{(d-2)\ell+\theta} \mid a_{d+j}, \quad p^{\theta} \mid s_{d+1+j}.
\end{equation*}
This completes the induction and proves the assertion of the theorem. 
\end{proof}

\medskip
In general, having a root in $\Z_p$ is certainly not necessary for a polynomial to factor in $\Z[[x]]$. For instance, 
\begin{equation*}
 f(x)=49+98x+63x^2+14x^3+x^4
\end{equation*}
is reducible in $\Z[[x]]$, namely $f(x)=(7+7x+x^2)^2$, but it has no roots in $\Z_7$.

However, as we discuss below, if the polynomial
\begin{equation*}
f(x)=p^n + p^m\gamma_1 x + \gamma_2 x^2 +\cdots+ \gamma_dx^d, \quad \gamma_d\not=0,
\end{equation*}
is as in Theorem~\ref{case:n<=2m}, with the additional condition that $\gcd(p,\gamma_2,\gamma_3)=1$, then its reducibility in $\Z[[x]]$ gives the existence of a root in $p\Z_p$. Note that if $n\le 2m$ and $f(x)$ is reducible in $\Z[[x]]$, then there exist $\ell$, $a_k$, $b_k\in \Z$, $1\le\ell\le n/2$, such that
\begin{equation}\label{eq:factors}
f(x)= \big(p^\ell+a_1x+a_2x^2+\cdots\big)\big(p^{n-\ell}+b_1x+b_2x^2+\cdots\big).
\end{equation}

\begin{proposition}\label{rootg(x)}
Let $f(x)$ be a polynomial of the form \eqref{def:polynomial} with $n\le 2m$, and such that $\gcd(p,\gamma_2,\gamma_3)=1$. If $f(x)$ factors as above, then $f(x)$ has a root $\varrho\in \Z_p$ with $v_p(\varrho)=\ell$.
\end{proposition}

In order to prove this proposition, we need the following lemma.

\begin{lemma}\label{lemma:choosing}
If $f(x)$ as above is reducible in $\Z[[x]]$, then the factorization \eqref{eq:factors} can be arranged such that $\gcd(p,a_1)=1$ and $a_2=\cdots=a_k=0$ for any $k\ge 2$.
\end{lemma}
\begin{proof}
From equation \eqref{eq:factors} we deduce the following equalities:
\begin{align*}
p^m\gamma_1 &=p^\ell(b_1+p^{n-2\ell}a_1), \\
\gamma_2 &=p^\ell(b_2+p^{n-2\ell}a_2)+a_1b_1,\\
\gamma_3 &=p^\ell(b_3+p^{n-2\ell}a_3)+a_1b_2+a_2b_1,
\end{align*}
which imply that either $\gcd(p,a_1)=1$ or $\gcd(p,b_1)=1$ (since $\gcd(p,\gamma_2,\gamma_3)=1$).

If $2\ell<n$, then $n\le 2m$ implies $\ell<m$, and so $p \mid b_1$. Hence $\gcd(p,a_1)=1$. 

If $2\ell=n$, then the two factors in \eqref{eq:factors} are exchangeable and we can assume that $a_1$ is the one relatively prime to $p$. 

We will proceed by induction in $k\ge 2$, assuming that $\gcd(p,a_1)=1$.  

Base step ($k=2$): Since $\gcd(p,a_1)=1$, there exist $u_1,u_2\in\Z$ such that
\begin{equation*}
 u_1 a_1  + u_2 p^\ell=-a_2.
\end{equation*}
If we write $f(x)=A(x)B(x)$ with $A(x),B(x)\in\Z[[x]]$ as in \eqref{eq:factors}, then
\begin{equation*}
f(x)= \big((1+u_1x+u_2x^2)A(x)\big)\big((1+u_1x+u_2x^2)^{-1}B(x)\big).
\end{equation*}
Since $(1+u_1x+u_2x^2)^{-1}B(x)\in\Z[[x]]$ and
\begin{align*}
 (1+u_1x+u_2x^2)A(x)&=(1+u_1x+u_2x^2)\big(p^\ell+a_1x+a_2x^2+\cdots\big) \\
 &= p^\ell+(a_1+u_1p^\ell)x+(a_2+u_1 a_1+u_2 p^\ell)x^2+\cdots \\
 &= p^\ell+c_1x+c_3x^3+\cdots,
\end{align*}
we now have a factorization of $f(x)$ with the desired properties. 

Inductive step: Suppose that \eqref{eq:factors} can be written as $f(x)=A_k(x) B(x)$ with
\begin{align*}
A_k(x) &= p^\ell+a_1x+a_{k+1}x^{k+1}+a_{k+2}x^{k+2}+\cdots, \\
B(x) &= p^{n-\ell}+b_1x+b_2x^2+\cdots.
\end{align*}
Since $\gcd(p,a_1^k)=1$, there exist $u_1, u_{k+1}\in\Z$ such that
\begin{equation}\label{k+1coeff}
 (-1)^{k-1}u_1 a_1^k + u_{k+1} p^{k\ell}=-a_{k+1}p^{(k-1)\ell}.
\end{equation}
In particular, $p^{(k-1)\ell}\mid u_1$.  If we let $u_2=-p^{-\ell}u_1a_1$, then
\begin{equation*}
 u_1 a_1 + u_{2} p^{\ell}= 0 \;\text{ and }\; p^{(k-2)\ell}\mid u_2.
\end{equation*}
For $j=3,\dots,k$, we recursively define $u_j=-p^{-\ell}u_{j-1}a_1$ and obtain
\begin{equation*}
 u_{j-1} a_1 + u_{j} p^{\ell}= 0 \;\text{ with }\; p^{(k-j)\ell}\mid u_j.
\end{equation*}
Finally, we let $U(x)=1+u_1x+u_2x^2+\cdots+u_{k+1}x^{k+1}$ and write
\begin{equation*}
f(x)= \big(U(x) A_k(x)\big)\big(U(x)^{-1}B(x)\big).
\end{equation*}
Observe that $U(x)^{-1}B(x)\in\Z[[x]]$. Moreover,
\begin{align*}
 U(x)A_k(x)&=(1+u_1x+\cdots+u_{k+1}x^{k+1})\big(p^\ell+a_1x+a_{k+1} x^{k+1}+\cdots\big) \\
 &= p^\ell+(a_1+u_1p^\ell)x+(u_1 a_1+u_2 p^\ell)x^2+(u_2 a_1+u_3 p^\ell)x^3+ \\
 & \qquad\; +(u_{k-1} a_1+u_k p^\ell)x^k + (a_{k+1}+u_k a_1 + u_{k+1} p^{\ell})x^{k+1} + \cdots 
\end{align*}
By construction, the coefficients of $x^2,\dots,x^k$ are all zero, and 
\begin{equation*}
 u_k a_1=(-1)^{k-1}p^{-(k-1)\ell}u_1a_1^{k}.
\end{equation*}
Thus, by \eqref{k+1coeff}, the coefficient of $x^{k+1}$ is also zero and we get
\begin{equation*}
 U(x)A_k(x) = p^\ell+c_1x+c_{k+2}x^{k+2}+\cdots
\end{equation*}
as desired.
\end{proof}

\begin{proof}[Proof of Proposition~\ref{rootg(x)}]
Let $k\in\N$ with $k\ge d$. By Lemma~\ref{lemma:choosing} we can assume that $f(x)$ factors as in \eqref{eq:factors} with $\gcd(p,a_1)=1$ and $a_2=\dots = a_d=\dots=a_{k}=0$. Note that the choice of $a_1$ depends on $k$. Using the notation $s_j=b_j+p^{n-2\ell}a_j$, we then get the equations
\begin{align*}
p^m\gamma_1 &=p^\ell s_1, \\
\gamma_2 &=p^\ell s_2+a_1(p^{m-\ell}\gamma_1-p^{n-2\ell}a_1), \\
\gamma_j &=p^\ell s_{j}+a_1s_{j-1} \text{ for } 2\le j\le d.
\end{align*}
Solving for $s_d$ in terms of $a_1$ and the coefficients of $f(x)$, these equations give the identity $g_d(a_1)=p^{(d-1)\ell}s_d$, where $g_d(x)$ is the polynomial 
\begin{equation*}
 \qquad g_d(x)=\frac{(-1)^d x^d}{p^{2\ell}}f(-p^\ell/x). \quad \text{(cf.  proof of Theorem~\ref{case:n<=2m})}
\end{equation*}

We also consider the next block of equations
\begin{align*}
0 &= p^\ell s_{d+1} +a_1s_d\\[-1ex]
& \;\; \vdots \\[-1ex]
0 &= p^\ell s_{k+1} +a_1s_{k}.
\end{align*}
Using that $\gcd(p,a_1)=1$, the last equation then gives $p^\ell\mid s_k$. This implies
\begin{equation*}
 p^{2\ell}\mid s_{k-1}\implies p^{3\ell}\mid s_{k-2}\implies \cdots \implies p^{(k-d+1)\ell}\mid{s_d},  
\end{equation*}
and therefore $g_d(a_1)\equiv 0 \pmod {p^{k\ell}}$. 

In conclusion, for every $k\ge d$ we have an integer $a_{1,k}$ such that 
\begin{equation*}
 g_d(a_{1,k})\equiv 0 \!\pmod {p^{k\ell}}.
\end{equation*}

By Lemma~\ref{rootinZp}, there exists $r\in\Z_p$ such that $g_d(r)=0$. Moreover, $v_p(a_{1,k})=0$ implies $v_p(r)=0$. In particular, $1/r\in\Z_p$. Finally, because of the relation between $f(x)$ and $g_d(x)$, it follows that $\varrho=-p^\ell/r\in\Z_p$ is a root of $f(x)$ with $v_p(\varrho)=\ell$.
\end{proof}

%%%%%%%%%%%%%%%%%%%%%%%%%%%%%%%%%%%%%%%%%%%%%%%%%%%%%%%%%%%%%%%%
\section{Examples and further remarks}
\label{sec:further}

In this section, we intend to illustrate our results for the special cases of quadratic and cubic polynomials. We also discuss some particular  aspects of the factorization of integer polynomials over $\Z[[x]]$ and $\Z_p[x]$. 

Once again, the only difficulty is to understand the factorization properties of polynomials of the form \eqref{def:polynomial}. 
Thus we consider
\begin{equation}\label{def:deg<=3}
f(x)=p^n + p^m\gamma_1 x + \gamma_2 x^2 +\gamma_3 x^3, \quad |\gamma_2|+|\gamma_3|\not=0,
\end{equation}
with $p$ prime, $n\ge 2$, $m\geq 1$, $\gcd(p,\gamma_1)=1$ or $\gamma_1=0$, and $\gcd(p,\gamma_2,\gamma_3)=1$.  

If $n>2m$, everything is said in Proposition~\ref{case:n>2m}.

If $n\le 2m$, we can combine Proposition~\ref{prop:doubleroot}, Theorem~\ref{case:n<=2m}, and Proposition~\ref{rootg(x)} to formulate the following result. 

\begin{theorem}\label{thm:deg<=3}
Let $f(x)$ be a polynomial of the form \eqref{def:deg<=3} with $n\le 2m$ or $\gamma_1=0$. Then $f(x)$ admits a factorization 
\begin{equation*}
 f(x)=\big(p^\ell+a_1 x+a_2 x^2+\cdots\big)\big(p^{n-\ell}+b_1 x+b_2 x^2+\cdots\big) \text{ in } \Z[[x]]
\end{equation*}
with $1\le \ell\le n/2$ if and only if it has a root $\varrho\in\Z_p$ with $v_p(\varrho)=\ell\ge 1$.
\end{theorem}

As the polynomials $(7+7x+x^2)^2$  and $(7+7x+x^2)(7+7x+x^3)$ show, Theorem~\ref{thm:deg<=3} is not necessarily valid for polynomials of degree higher than 3. Observe that neither $7+7x+x^2$ nor $7+7x+x^3$ have roots in $\Z_7$. 

On the other hand, under the assumption that $f(x)$ has a root in $\Z_p$, it is indeed necessary to assume that its $p$-adic valuation be positive. For instance, the polynomial $7+21x+15x^2+x^3=(1+x)(7+14x+x^2)$ has $-1$ as root in $\Z_7$, but it does not admit a proper factorization in $\Z[[x]]$ since $1+x$ is a unit and $7+14x+x^2$ is irreducible in $\Z[[x]]$. Note that $v_7(-1)=0$.

As discussed in Section~\ref{sec:preliminaries}, the conditions $n\ge 2$ and $m\ge 1$ are necessary for Theorem~\ref{thm:deg<=3} to hold. In fact, any power series with $f(0)=1$ is a unit in $\Z[[x]]$, and any power series for which $f(0)$ is a prime number, is irreducible in $\Z[[x]]$. Moreover, if $\gcd(p,\gamma_1)=1$, any power series of the form $p^n+\gamma_1x+\cdots$, is irreducible in $\Z[[x]]$. However, the polynomial $p^n+\gamma_1x+\gamma_2x^2$ is reducible in $\Z_p[x]$ for every $\gamma_1,\gamma_2\in\Z$ with $\gcd(p,\gamma_1)=1$.

In other words, in general, the reducibility of polynomials in $\Z[[x]]$ does not follow from their reducibility in $\Z[x]$ or $\Z_p[x]$.  However, J.-P.~B{\'e}zivin \cite{Bezivin} claims that, in certain cases, the reducibility in $\Z[[x]]$ of a polynomial $p^n+\sum_{k=1}^d \!\alpha_k x^k$ is equivalent to their reducibility in $\Z_p[x]$. 

\medskip
We finish this section with a brief discussion about the factorization of quadratic and cubic polynomials in the presence of a multiple $p$-adic root.
\smallskip

\textbf{Multiple roots.}
As mentioned at the beginning of Section~\ref{sec:Main}, if $f(x)$ has a multiple root in $\Z_p$, then the factorization of $f(x)$ in $\Z[[x]]$ can be achieved with polynomial factors. We now proceed to illustrate how this factorization looks like for the cases at hand. Recall that if $f(x)$ has a multiple root, then its discriminant must be zero.

First, if $f(x)$ in \eqref{def:deg<=3} is quadratic and has a double root, then $p^{2m}\gamma_1^2-4p^n\gamma_2=0$. Hence $n=2m$, $\gamma_1$ is an even number, and
\[ f(x)=p^n + p^m\gamma_1 x + \gamma_2 x^2 = (p^m + \tfrac{\gamma_1}{2} x)^2. \]

Consider now $f(x)=p^n + p^m\gamma_1 x + \gamma_2 x^2 +\gamma_3 x^3$ as in \eqref{def:deg<=3} with $\gamma_3\not=0$, $n\le 2m$, and suppose that it has a double root, say $\varrho\in\Z_p$. Thus
\begin{equation*}
 p^n + p^m\gamma_1 \varrho + \gamma_2 \varrho^2 +\gamma_3 \varrho^3 = 0, \quad  
 p^m\gamma_1 + 2\gamma_2 \varrho + 3\gamma_3 \varrho^2 = 0,
\end{equation*}
and so $p$ must divide $\varrho^2(\gamma_2+\gamma_3\varrho)$ and $\varrho(2\gamma_2 + 3\gamma_3\varrho)$. Since $\gcd(p,\gamma_2,\gamma_3)=1$, we then conclude that $p$ divides $\varrho$. Let $\ell=v_p(\varrho)\ge 1$. Note that $n\le 2m$ implies $\ell\le m$ and $2\ell\le n$. Moreover, using that $9\gamma_3 f(\varrho)-(3\gamma_3\varrho+\gamma_2)f'(\varrho)=0$, we obtain
\begin{equation*}
\varrho = \frac{p^m\gamma_1\gamma_2-9p^n\gamma_3}{2(3p^m\gamma_1\gamma_3-\gamma_2^2)}\in \Q.
\end{equation*}

Write $\varrho=-b/a$ with $a,b\in\Z$ such that $\gcd(a,b)=1$. Then $a\mid \gamma_3$ and $b\mid p^n$, hence $b=p^\ell$ and $\varrho=-p^\ell/a$. We conclude that $f(x)$ admits the factorization 
\begin{equation*}
 f(x)=(p^\ell+ax)\big(p^{n-\ell} + (p^{m-\ell}\gamma_1-p^{n-2\ell}a)x + \tfrac{\gamma_3}{a} x^2\big),
\end{equation*}
which is a proper factorization of $f(x)$ in $\Z[[x]]$.

%%%%%%%%%%%%%%%%%%%%%%%%%%%%%%%%%%%%%%%%%%%%%%%%%%%%%%%%%%%%%%%%


\begin{thebibliography}{10}

\bibitem{Bezivin}
B\'{e}zivin, Jean-Paul, Unpublished communication, 2010.

\bibitem{BiGi}
Birmajer, Daniel, and Gil, Juan B., \emph{Arithmetic in the ring of formal power series with integer coefficients}, Amer. Math. Monthly \textbf{115} (2008), no. 6, 541--549. 

\bibitem{BGW}
Birmajer, Daniel; Gil, Juan B., and Weiner, Michael D., \emph{Factorization of quadratic polynomials in the ring of formal power series over $\Z$}, J. Algebra Appl. \textbf{6} (2007), no. 6, 1027--1037. 

\bibitem{Katok}
Katok, Svetlana, \emph{$p$-adic analysis compared with real}, Mathematics Advanced Study Semesters
(Student Mathematical Library, 37). American Mathematical Society, Providence, RI, 2007. 

\bibitem{Serre}
Serre, Jean-Pierre, \emph{A course in arithmetic}, Graduate Texts in  Mathematics, No. 7, Springer Verlag, New York-Heidelberg, 1973.
 
\end{thebibliography}
\end{document}